\documentclass[10pt]{article}

\usepackage[dvips]{graphicx}
\usepackage{rotating}
\usepackage{graphics}
\usepackage{epsfig}
\usepackage{amsmath}
\usepackage{amssymb}
\usepackage{amsthm}
\usepackage{amsfonts}
\usepackage{latexsym}
\usepackage{psfrag, float}
\usepackage{multirow}
\usepackage{bezier}
\usepackage{curves}

\title{Countable homogeneous lattices}
\author{A. Abogatma and J.K. Truss} 
\date{Department of Pure Mathematics,
          University of Leeds,
          Leeds LS2 9JT, UK, pmtjkt@leeds.ac.uk$^1$.}%\footnotemark}
\begin{document}
\maketitle 
\newtheorem{lemma}{Lemma}[section]
\newtheorem{theorem}[lemma]{Theorem}
\newtheorem{corollary}[lemma]{Corollary}
\newtheorem{definition}[lemma]{Definition}
\newtheorem{remark}[lemma]{Remark}
\newtheorem{construction}[lemma]{Construction}
\setcounter{footnote}{1}\footnotetext{This paper forms part of the first author's PhD thesis at the University of Leeds. 2010 Mathematics Classification 06A99, 03G20}
\newcounter{number}

\begin{abstract}
We show that there are uncountably many countable homogeneous lattices. We give a discussion of which such lattices can be 
modular or distributive. The method applies to show that certain other classes of structures also have uncountably many 
homogeneous members.
\end{abstract}

\section{Introduction}\label{intro}

A countable structure $\cal A$ is said to be {\em homogeneous} (sometimes also called `ultrahomogeneous') if any isomorphism 
between finitely generated substructures extends to an automorphism. In the literature, this notion is usually applied to 
relational structures, for instance in \cite{lachlan}, in which case `finitely generated' can be replaced by `finite'. The 
general Fra\"iss\'e theory of homogeneous structures however applies even when there are functions in the signature, with 
appropriate modifications. He gave necessary and sufficient conditions for a countable structure to be homogeneous, in terms 
of its {\em age}, being the family of finitely generated structures $\cal C$ which are isomorphic to a substructure of 
$\cal A$. Thus Fra\"iss\'e's Theorem (see \cite{hodges} for instance) says that $\cal C$ is equal to the age of some countable 
homogeneous structure if and only if it has the following properties:

$\cal C$ has only countably many members, up to isomorphism,

any finitely generated substructure of a member of $\cal C$ lies in $\cal C$,

any structure isomorphic to a member of $\cal C$ lies in $\cal C$,

the joint embedding property JEP: any two members of $\cal C$ can be embedded in another member of $\cal C$,

the amalgamation property AP: if $A$, $B$, and $C \in {\cal C}$, and embeddings $p_1: A \to B$, $p_2: A \to C$ are given, then 
there are $D \in {\cal C}$, and embeddings $p_3: B \to D$, $p_4: C \to D$ such that `the diagram commutes', meaning that 
$p_3p_1 = p_4p_2$. 

Any class having these five properties is referred to as an {\em amalgamation class}, and so Fra\"iss\'e's Theorem reduces 
the search for countable homogeneous structures to that for amalgamation classes. (Note that there is a slightly different use 
of the term `amalgamation class' in the literature, namely the amalgamation class of a class of structures comprises all its 
members over which all amalgamations can be performed.) Two standard examples of amalgamation classes are the family of all 
finite linear orders, where the corresponding homogeneous structure is the ordered set of rational numbers, and the class of 
all finite partial orders, where the homogeneous structure is the `generic partial order' (which is one of the structures in 
Schmerl's list \cite{schmerl} of all the countable homogeneous partial orders).

In this paper the structures we consider are lattices. These are partially ordered sets in which any two elements $a$ and $b$ 
have a least upper bound $a \vee b$ and a greatest lower bound $a \wedge b$, referred to as their `join' and `meet' 
respectively. We work in the signature $\{\vee, \wedge\}$ (in which $\le$ is quantifier-free definable).   

As far as we are aware, the following are the countable homogeneous lattices known up till now: the trivial (one-point) 
lattice, the Fra\"iss\'e limit of the class of all finite lattices, the Fra\"iss\'e limit of the class of all finite 
distributive lattices, and the ordered set of rational numbers. J\'onsson showed in \cite{jonsson} that the class of all 
lattices forms an amalgamation class, and the same proof shows that the class of all finite lattices does too; Pierce 
\cite{pierce} gave the analogous result for the class of distributive lattices. In fact, the only {\em varieties} of lattices
which have the amalgamation property correspond to the following three cases: the varieties of all lattices, all distributive 
lattices, and all trivial lattices. See \cite{day}. 

This remarkable result does not however seem to preclude the existence of other countable homogeneous lattices, and the main 
result of this paper shows that there are indeed uncountably many such. Since the ones we construct are all non-modular, the
next natural question is to elucidate the situation with regard to modular lattices. The conjecture would be that any 
countable homogeneous modular lattice is also distributive; alternatively one can examine the question as to whether the
class of all finite modular lattices is an amalgamation class, or, more generally, whether there exists a class of finite 
modular lattices, not all distributive, which is an amalgamation class, or failing this, whether such might exist on replacing 
`finite' by `finitely generated'. As remarked in \cite{day}, it seems hard to find concrete examples of the failure of 
amalgamation (the authors say that since their proof requires only that $N_5$ lies in the variety, it would be of interest 
to have an elementary proof, as opposed to \cite{gratzer2} that if a variety satisfies the amalgamation property and $M_3$ 
lies in the variety, then so does $N_5$, in the standard notation).

\section{The main proof} 

We begin by recalling the proof of the amalgamation property for the family of all lattices, due to J\'onsson, which 
specializes to various other classes. An isomorphism from a substructure of a structure $\cal A$ to itself is referred to as 
a {\em partial automorphism}.

\begin{lemma} \label{2.1} The following families of lattices all have the amalgamation property:

(i) the family of finite lattices,

(ii) the family of countable lattices,

(iii) the family of finitely generated lattices.
\end{lemma} 

\begin{proof} Let $A$, $B$, $C$ be lattices in one of these classes, and $p_1: A \to B$, $p_2: A \to C$  be given embeddings. 
In verifying that this diagram can be amalgamated, we may assume, by replacing $B$ and $C$ by isomorphic copies, that $A$ is a 
sublattice of both $B$ and $C$, and that $A = B \cap C$. In the finite case, the amalgam may then be taken to be the 
Dedekind--MacNeille completion $D$ of the partial order amalgam $B \cup C$ of $B$ and $C$, and in the other cases, it is taken 
to be the sublattice of $D$ generated by $B \cup C$. 
\end{proof}

We remark that on the basis of this result, and the already known examples, we can so far identify the following five 
countable homogeneous lattices, which are all easily seen to be non-isomorphic: the trivial lattice, the rational numbers 
under the usual relation, the countable generic distributive lattice, and the Fra\"iss\'e limits of the classes of finite 
lattices (called the `generic locally finite lattice') and of all countable lattices (since the Fra\"iss\'e limit of the class 
of all finitely generated lattices equals that of the class of all countable lattices, as they have the same age). In the 
remainder of the paper we show that there are however $2^{\aleph_0}$ further examples. We formulate this result in a rather 
more general context. We are grateful to the referee for pointing out the generalization.

\begin{lemma} \label{2.2} Let $\cal C$ be an amalgamation class of finite or countable structures. Then for any 
${\cal A} \in {\cal C}$ and any partial automorphism $p$ of $\cal A$, there are an extension $\cal B$ of $\cal A$ in $\cal C$ 
and an extension $\theta$ of $p$ to a partial automorphism of $\cal B$ such that 
${\cal A} \subseteq {\rm dom} \, \theta \cap {\rm range} \, \theta$. \end{lemma} 

\begin{proof} It suffices to deal with the domain as we can use a similar argument for the range. Let us consider
the three structures dom $p$, $\cal A$, and $\cal A$, and embed dom $p$ into the first copy of $\cal A$ by inclusion $i$, and
into the second by $p$. By the amalgamation property there is ${\cal B} \in {\cal C}$ and there are embeddings $q, q'$
of the first and second copies of $\cal A$ into $\cal B$ so that $qi = q'p$. By replacing $\cal B$ by a copy we may assume that
$q'$ is inclusion, and it follows that $q$ is an extension of $p$ whose domain contains $\cal A$. Repeating this argument on 
$q^{-1}$ leads to an extension $\theta$ of $q$ whose range contains $\cal A$.

\end{proof}

\begin{lemma} \label{2.3} Let $\cal C$ be an amalgamation class of finite or countable structures. Then any member of $\cal C$ 
is a substructure of some countable homogeneous structure which is the union of a countable chain of members of $\cal C$. \end{lemma} 

\begin{proof} Let $\cal A$ be the given countable structure. We construct the desired countable homogeneous structure as the 
union ${\cal A}^*$ of a chain of structures ${\cal A} = {\cal A}_0 \le {\cal A}_1 \le {\cal A}_2 \le \ldots$ in $\cal C$. At 
each stage we shall also have a countable list $P_n$ of partial automorphisms of ${\cal A}_n$ which are to be extended to 
automorphisms of ${\cal A}^*$. Let $\theta: \omega \to \omega^2$ be a bijection to help with `book-keeping' such that if
$\theta(n) = (i, j)$, then $j \le n$.

Let $P_0$ be the family of partial automorphisms of ${\cal A}_0$ whose domain (and hence also range) is finitely generated.

In general, assume that ${\cal A}_n$ in $\cal C$ and countable $P_0, P_1, \ldots, P_n$ have been chosen, and suppose that the 
members of each $P_j$ are arranged in a list of length $\omega$. Let $\theta(n) = (i, j)$ so that $j \le n$. We apply Lemma 
\ref{2.2} to the $i$th member $p$ of $P_j$, finding an extension of it to a partial automorphism $q$ of a structure 
${\cal A}_{n+1} \in {\cal C}$ whose domain and range contain ${\cal A}_n$. We let $P_{n+1}$ be the family of all partial 
automorphisms of ${\cal A}_{n+1}$ whose domain is finitely generated together with $q$.

Each partial automorphism $p$ of ${\cal A}^*$ whose domain is finitely generated is then a partial automorphism of some 
${\cal A}_n$, so lies in $P_n$, and at infinitely many subsequent stages is extended so that the union of all these extensions 
is an automorphism of ${\cal A}^*$ extending $p$. Thus ${\cal A}^*$ is homogeneous.  \end{proof}

\begin{theorem} \label{2.4} Let $\cal C$ be an amalgamation class of finite or countable structures, and suppose that there 
are ${\cal Z} \le {\cal A}$, both in $\cal C$, such that $\cal A$ is finitely generated and $\cal Z$ has $2^{\aleph_0}$ 
self-embeddings. Then $\cal C$ contains at least $2^{\aleph_0}$ isomorphism types of finitely generated structures.
\end{theorem}

\begin{proof} For each self-embedding $\theta$ of $\cal Z$, let us apply the amalgamation property to amalgamate two
copies ${\cal A}_1$ and ${\cal A}_2$ of $\cal A$ over $\cal Z$ viewed as a substructure of ${\cal A}_1$ via the inclusion map, 
and as a substructure of ${\cal A}_2$ via $\theta$. Then this amalgam is generated by the union of generating sets for 
${\cal A}_1$ and ${\cal A}_2$, so is finitely generated.

We must show that this construction results in $2^{\aleph_0}$ pairwise non-isomorphic structures. For this, let us say that 
$\cal C$ {\em exhibits} a self-embedding $\theta: {\cal Z} \to {\cal Z}$ if for some embeddings $\varphi_1$ and $\varphi_2$ of 
$\cal A$ into $\cal C$, $\cal C$ is generated by $\varphi_1{\cal A} \cup \varphi_2{\cal A}$ and for all $z \in {\cal Z}$, 
$\varphi_2 \theta(z) = \varphi_1(z)$. Then $\cal B$ resulting from the above construction exhibits $\theta$. 
Furthermore, given any $\cal C$, we see that it can only exhibit countably many self-embeddings $\theta$. For if $\theta$ is 
exhibited, and $\varphi_1$ and $\varphi_2$ are the corresponding maps, then first, there are only countably many possibilities 
for $\varphi_1$ and $\varphi_2$ as they are determined by their actions on a finite generating set, and there only $\aleph_0$
possibilities for where they are mapped, and in view of the equation $\varphi_2 \theta(z) = \varphi_1(z)$, $\theta$ is 
determined from $\varphi_1, \varphi_2$. Now let $\kappa$ be the number of non-isomorphic structures arising from the 
construction. Then the total number of self-embeddings of $\cal Z$ exhibited is at most $\kappa \cdot \aleph_0$. But we are
told that the number of functions which can be exhibited equals $2^{\aleph_0}$. Therefore 
$\kappa \cdot \aleph_0 = 2^{\aleph_0}$ which implies that $\kappa  = 2^{\aleph_0}$.    \end{proof}

\begin{corollary} \label{2.5} There are $2^{\aleph_0}$ non-isomorphic finitely generated lattices.
\end{corollary}

\begin{proof}  According to \cite{dilworth},  the free lattice on 3 generators has an infinite chain in order-type $\mathbb Z$,
$\{z_n: n \in {\mathbb Z}\}$ say, and we let $\cal Z$ be the sublattice of the free lattice with this domain. Then there are
$2^{\aleph_0}$ self-embeddings of $\cal Z$ into $\cal Z$, so the result follows at once from the theorem.  \end{proof}

\begin{corollary} \label{2.6}  There are $2^{\aleph_0}$ non-isomorphic countable homogeneous lattices.
\end{corollary}

\begin{proof} By Corollary \ref{2.5} there are $2^{\aleph_0}$ non-isomorphic finitely generated lattices, and these are all 
countable. By Lemma \ref{2.3}, each is a sublattice of a countable homogeneous lattice. Each countable lattice only has 
countably many finitely generated sublattices, and therefore there must be $2^{\aleph_0}$ of these countable homogeneous 
lattices which are not isomorphic.   \end{proof}

Although our main focus is on lattices, the referee has pointed out that the main result can be applied in many more situations, and we now describe some of these.

\begin{corollary} \label{2.7}  There are $2^{\aleph_0}$ non-isomorphic countable homogeneous groups. \end{corollary}

\begin{proof} Using Lemma \ref{2.3} and Theorem \ref{2.4} it is clear that we just need to identify suitable $\cal A$ and
$\cal Z$ to which Theorem \ref{2.4} can be applied. For this we can take $\cal A$ to be a free group on 2 generators, and 
$\cal Z$ a subgroup which is a free group of infinite rank.  \end{proof}

\begin{corollary} \label{2.8} Let $R$ be a countable ring for which there is a left invertible $m \times n$ matrix, where 
$m < n$. Then there are $2^{\aleph_0}$ non-isomorphic countable homogeneous $R$-modules. \end{corollary}

\begin{proof} This is similar to the previous corollary, where we let $\cal A$ be an $m$-generator free $R$-module, 
and $\cal Z$ be an infinitely generated free $R$-module inside $\cal A$. The fact that such $\cal Z$ exists follows from
the existence of the stated left invertible matrix. Further background to the existence of rings with the stated property 
may be found for instance in \cite{bergman1}. \end{proof}

Finally we remark that a more general version of these results would state that any variety of algebras having the 
amalgamation property and such that the $\aleph_0$-generated free algebra is embeddable in a finitely generated algebra,
contains $2^{\aleph_0}$ countable homogeneous members. Further examples can be derived from Corollary 12(iii) of 
\cite{bergman2}, which states the following: If $V$ is a variety of finitary algebras which is generated by the free algebra 
$F_V(x_1, x_2, \ldots, x_m)$ on $m$ generators, and this contains a free algebra on $> m$ generators, then it also contains a 
free algebra on $\aleph_0$ generators.

\end{document}